\documentclass[11pt]{article}
\usepackage{amsmath}
\usepackage{amssymb}
\usepackage{amsthm}
\usepackage{enumerate}
\usepackage[colorlinks]{hyperref}
\usepackage{enumitem}
\usepackage[margin=1in]{geometry}
\usepackage{stmaryrd}
\usepackage{pgfplots, pgfplotstable, booktabs, colortbl, array,multirow}

\theoremstyle{plain}
\newtheorem{theorem}{Theorem}[section]

\newtheorem{proposition}[theorem]{Proposition}

\theoremstyle{definition}

\newtheorem{remark}{Remark}[section]

\DeclareMathOperator{\dv}{div}

\newcommand{\maketable}[1]{%
\pgfplotstabletypeset[
header=false,
font=\small,
clear infinite,
every head row/.style={before row=\hline,after row=\hline},
every last row/.style={after row=\hline},
every row no 4/.style={before row=\hline},
every row no 8/.style={before row=\hline},
every row no 12/.style={before row=\hline},
create on use/newcol/.style={
        create col/set list={NaN,0,NaN,NaN,NaN,1,NaN,NaN,NaN,2,NaN,NaN,}
},
columns={newcol,0,1,2,3,4,5,6},
columns/newcol/.style={column type/.add={|}{},column name={$s$}},
columns/0/.style={sci zerofill,column type/.add={|}{|},column name={$h^{-1}$}},
columns/1/.style={dec sep align={c|},sci,sci 10e,sci zerofill,precision=2,column type/.add={}{|},column name={$\|u_h-u\|_{L^2(\Omega)}$}},
columns/3/.style={dec sep align={c|},sci,sci 10e,sci zerofill,precision=2,column type/.add={}{|},column name={$\|\rho_h-\rho\|_{L^2(\Omega)}$}}, 
columns/5/.style={dec sep align={c|},sci,sci 10e,sci zerofill,precision=2,column type/.add={}{|},column name={$\|p_h-p\|_{L^2(\Omega)}$}}, 
columns/2/.style={dec sep align={c|},fixed zerofill,precision=2,column type/.add={}{|},column name={Rate}},
columns/4/.style={dec sep align={c|},fixed zerofill,precision=2,column type/.add={}{|},column name={Rate}},
columns/6/.style={dec sep align={c|},fixed zerofill,precision=2,column type/.add={}{|},column name={Rate}},
]
{#1}
}

\newcommand{\maketabledt}[1]{%
\pgfplotstabletypeset[
header=false,
font=\small,
clear infinite,
every head row/.style={before row=\hline,after row=\hline},
every last row/.style={after row=\hline},
every row no 4/.style={before row=\hline},
every row no 8/.style={before row=\hline},
columns={0,1,2,3,4,5,6},
columns/0/.style={sci zerofill,column type/.add={|}{|},column name={$\Delta t^{-1}$}},
columns/1/.style={dec sep align={c|},sci,sci 10e,sci zerofill,precision=2,column type/.add={}{|},column name={$\|u_h-u\|_{L^2(\Omega)}$}},
columns/3/.style={dec sep align={c|},sci,sci 10e,sci zerofill,precision=2,column type/.add={}{|},column name={$\|\rho_h-\rho\|_{L^2(\Omega)}$}}, 
columns/5/.style={dec sep align={c|},sci,sci 10e,sci zerofill,precision=2,column type/.add={}{|},column name={$\|p_h-p\|_{L^2(\Omega)}$}}, 
columns/2/.style={dec sep align={c|},fixed zerofill,precision=2,column type/.add={}{|},column name={Rate}},
columns/4/.style={dec sep align={c|},fixed zerofill,precision=2,column type/.add={}{|},column name={Rate}},
columns/6/.style={dec sep align={c|},fixed zerofill,precision=2,column type/.add={}{|},column name={Rate}}
]
{#1}
}

\begin{document}

\title{A Conservative Finite Element Method for the Incompressible Euler Equations with Variable Density}

\author{Evan S. Gawlik\thanks{\noindent Department of Mathematics,  University of Hawaii at Manoa, \href{egawlik@hawaii.edu}{egawlik@hawaii.edu}} \; and \; Fran\c{c}ois Gay-Balmaz\thanks{\noindent CNRS - LMD, Ecole Normale Sup\'erieure, \href{francois.gay-balmaz@lmd.ens.fr}{francois.gay-balmaz@lmd.ens.fr}}}

\date{}

\maketitle

\begin{abstract}
We construct a finite element discretization and time-stepping scheme for the incompressible Euler equations with variable density that exactly preserves total mass, total squared density, total energy, and pointwise incompressibility.  The method uses Raviart-Thomas or Brezzi-Douglas-Marini finite elements to approximate the velocity and discontinuous polynomials to approximate the density and pressure.  To achieve exact preservation of the aforementioned conserved quantities, we exploit a seldom-used weak formulation of the momentum equation and a second-order time-stepping scheme that is similar, but not identical, to the midpoint rule.  We also describe and prove stability of an upwinded version of the method.  We present numerical examples that demonstrate the order of convergence of the method.
\end{abstract}

\section{Introduction}

This paper considers the incompresible Euler equations with variable density on a bounded domain $\Omega \subset \mathbb{R}^n$, $n \in \{2,3\}$:
\begin{align}
\rho(\partial_t u + u \cdot \nabla u) &= -\nabla p, & \text{ in } \Omega \times (0,T), \label{velocity0} \\
\partial_t \rho + \dv (\rho u) &= 0, & \text{ in } \Omega \times (0,T), \label{density0} \\
\dv u &= 0, & \text{ in } \Omega \times (0,T), \label{incompressible0} \\
u \cdot n &= 0, & \text{ on } \partial\Omega \times (0,T), \label{BC} \\
u(0) = u_0, \, \rho(0) &= \rho_0, & \text{ in } \Omega. \label{IC}
\end{align}
The aim of this paper is to construct a finite element method for~(\ref{velocity0}-\ref{incompressible0}) that exactly conserves
\begin{equation} \label{conservationlaws}
\int_\Omega \rho \, dx,  \quad \int_\Omega \rho^2 \, dx, \quad \int_\Omega \rho u \cdot u \, dx,
\end{equation}
and the incompressibility constraint $\dv u = 0$ at the (spatially and temporally) discrete level.

Several authors have considered related tasks in the past.  Guermond and Quartapelle~\cite{guermond2000projection} have constructed a spatial discretization of the incompressible Navier-Stokes equations that preserves the three invariants~(\ref{conservationlaws}) in the limit of vanishing viscosity.  However, their temporal discretization does not preserve total energy nor total squared density.  Moreover, the incompressibility constraint is only satisfied in a weak sense, and their construction relies on an evolution equation for $\sqrt{\rho}u$ -- a quantity that lacks physical meaning.  
Another method that is closely related to ours is the $H(\dv)$-conforming finite element spatial discretization of the incompressible Euler equations with constant density studied in~\cite{guzman2016h,natale2017variational}.  When $\rho \equiv 1$, our spatial discretization reduces to the one used there, and our temporal discretization reduces to the midpoint rule used in~\cite{natale2017variational}.  Lastly, the form of the momentum equation that we discretize in our method bears some resemblance to the one used in~\cite{lee2018discrete} for the compressible rotating shallow water equations, but with $\partial_t (\rho u)$ appearing rather than $\partial_t u$.

To motivate our numerical method, we first use the identity
\[
\rho u \cdot \nabla u = \nabla(\rho u \cdot u) - u \times (\nabla \times (\rho u)) - (u \cdot \nabla \rho) u - \frac{1}{2} \rho \nabla ( u \cdot u)
\]
and equations~(\ref{density0}-\ref{incompressible0}) to write~(\ref{velocity0}) in the form
\begin{align}
\partial_t (\rho u) + \nabla (\rho u \cdot u) - u \times (\nabla \times (\rho u)) - \frac{1}{2}\rho \nabla(u \cdot u) &= -\nabla p, \label{velocity1}
\end{align}
Next, we multiply~(\ref{velocity1}),~(\ref{density0}), and~(\ref{incompressible0}) by test functions $v$, $\sigma$, and $q$, respectively, and integrate by parts.  Using the identity
\begin{align*}
\int_\Omega \left( \nabla (w \cdot u) - u \times (\nabla \times w) \right) \cdot v \, dx = \int_\Omega w \cdot (v \cdot \nabla u - u \cdot \nabla v) \, dx, \quad \text{ if } \dv u = 0, & \\
\text{ and } \left.u \cdot n\right|_{\partial\Omega} = 0, &
\end{align*}
we arrive at the following observation.  For every smooth vector field $v$ satisfying $\left. v \cdot n \right|_{\partial\Omega} = 0$ and every pair of smooth scalar fields $\sigma$ and $q$, the solution $(u,\rho,p)$ of~(\ref{velocity0}-\ref{IC}) satisfies
\begin{align}
\langle \partial_t(\rho u), v \rangle + a(\rho u, u, v) - \frac{1}{2} b(v,u \cdot u,\rho) &= \langle p, \dv v \rangle, \label{velocity} \\
\langle \partial_t \rho, \sigma \rangle - b(u,\sigma,\rho) &= 0, \label{density} \\
\langle \dv u, q \rangle &= 0, \label{incompressible}
\end{align}
where $\langle u, v \rangle = \int_\Omega u \cdot v \, dx$ for vector fields $u$ and $v$, $\langle f,g \rangle = \int_\Omega f g \, dx$ for scalar fields $f$ and $g$, and
\begin{align*}
a(w,u,v) &= \langle w, v \cdot \nabla u - u \cdot \nabla v \rangle, \\
b(w,f,g) &= \langle w \cdot \nabla f, g \rangle.
\end{align*}
Note that~(\ref{velocity}-\ref{incompressible}) can alternatively be derived from Hamilton's principle of least action; we refer the reader to~\cite{gawlik2019variational} for details.

One advantage of the formulation~(\ref{velocity}-\ref{incompressible}) is that the conserved quantities~(\ref{conservationlaws}) are straightforward consequences of two basic properties of the trilinear forms $a$ and $b$.  Namely, $a$ is alternating in its last two arguments,
\begin{equation} \label{aalternating}
a(w,u,v) = -a(w,v,u), 
\end{equation}
and $b$ is alternating in its last two arguments when its first argument is divergence-free:
\begin{align} \label{balternating}
b(w,f,g) &= -b(w,g,f) \text{ if } \dv w = 0 \text{ and } \left. w \cdot n \right|_{\partial\Omega} = 0.
\end{align}

With these properties in mind, we take $\sigma = 1$ in the density equation~(\ref{density}) to deduce conservation of total mass:
\[
\frac{d}{dt} \int_\Omega \rho \, dx = \langle \partial_t \rho, 1 \rangle = b(u,1,\rho) = 0.
\]
If instead we take $\sigma = \rho$ in~(\ref{density}) and use~(\ref{balternating}), we deduce conservation of total squared density:
\[
\frac{d}{dt} \frac{1}{2} \int_\Omega \rho^2 \, dx = \langle \partial_t \rho, \rho \rangle = b(u,\rho,\rho) = 0.
\]
Finally, taking $v=u$ in the momentum equation~(\ref{velocity}) gives conservation of total energy:
\begin{align*}
\frac{d}{dt} \frac{1}{2} \int_\Omega \rho u \cdot u \, dx 
&= \langle \partial_t(\rho u), u \rangle - \frac{1}{2} \langle \partial_t \rho, u \cdot u \rangle \\
&= \langle p, \dv u \rangle - a(\rho u, u, u) + \frac{1}{2}b(u, u \cdot u, \rho) - \frac{1}{2} \langle \partial_t \rho, u \cdot u \rangle \\
&= 0.
\end{align*}
Here, we have used the fact that $\dv u = 0$, $a$ is alternating in its last two arguments, and~(\ref{density}) holds.  Our numerical method will inherit these conservation laws by using discretizations of $a$ and $b$ that satisfy analogues of~(\ref{aalternating}) and~(\ref{balternating}).

\section{Spatial Discretization}

In this section, we propose a finite element spatial discretization of~(\ref{velocity}-\ref{incompressible}). 

Let $\mathcal{T}_h$ be a triangulation of $\Omega$.  We denote by $\mathcal{E}_h$ the set of interior $(n-1)$-dimensional faces in $\mathcal{T}_h$ (edges in two dimensions). For each integer $s \ge 0$ and each simplex $K \in \mathcal{T}_h$, we denote by $P_s(K)$ the space of polynomials of degree at most $s$ on $K$.  We denote
\[
H_0(\dv,\Omega) = \{u \in L^2(\Omega)^n \mid \dv u \in L^2(\Omega), \, u \cdot n = 0 \text{ on } \partial\Omega \}
\]
and
\[
\mathring{H}(\dv,\Omega) = \{u \in H_0(\dv,\Omega) \mid \dv u = 0\}.
\]

Our numerical method will make use of three approximation spaces: a space $U_h \subset H_0(\dv,\Omega)$ for the velocity $u$, a space $F_h \subset L^2(\Omega)$ for the density $\rho$, and a space $Q_h \subset L^2(\Omega)$ for the pressure $p$.  For the velocity, we use either the Raviart-Thomas space
\[
RT_s(\mathcal{T}_h) = \{ u \in H_0(\dv,\Omega) \mid \left. u \right|_K \in P_s(K)^n + x P_s(K), \, \forall K \in \mathcal{T}_h \} 
\]
or the Brezzi-Douglas-Marini space
\[
BDM_{s+1}(\mathcal{T}_h) = \{ u \in H_0(\dv,\Omega) \mid \left. u \right|_K \in P_{s+1}(K)^n, \, \forall K \in \mathcal{T}_h \},
\]
where $s \ge 0$ is an integer.  For the pressure, we use the discontinuous Galerkin space
\[
DG_s(\mathcal{T}_h) = \{ f \in L^2(\Omega) \mid \left. f \right|_K \in P_s(K), \, \forall K \in \mathcal{T}_h \}.
\]
For the density, we use $DG_m(\mathcal{T}_h)$, where $m \ge 0$ is an integer (not necessarily equal to $s$).  In summary,
\begin{align}
U_h &\in \{ RT_s(\mathcal{T}_h), BDM_{s+1}(\mathcal{T}_h) \}, \label{FEvelocity} \\
F_h &= DG_m(\mathcal{T}_h), \label{FEdensity} \\
Q_h &= DG_s(\mathcal{T}_h). \label{FEpressure}
\end{align}

We also denote by $W_h$ the (infinite-dimensional) space of vector fields on $\Omega$ that are piecewise smooth with respect to $\mathcal{T}_h$.

On an edge $e = K_1 \cap K_2 \in \mathcal{E}_h$, we denote the jump and average of a scalar function $f \in F_h$ by
\[
\llbracket f \rrbracket = f_1 n_1 + f_2 n_2, \quad \{f\} = \frac{f_1+f_2}{2},
\] 
where $f_i = \left. f \right|_{K_i}$, $n_1$ is the normal vector to $e$ pointing from $K_1$ to $K_2$, and similarly for $n_2$.  To define the jump and average of a vector field $u \in W_h$, we fix a choice of normal vector $n$ for each edge $e \in \mathcal{E}_h$.  Next, we determine $K_1$ and $K_2$ so that $e = K_1 \cap K_2$ and $n$ points from $K_1$ to $K_2$, and we define
\[
\llbracket u \rrbracket = u_1-u_2, \quad \{u\} = \frac{u_1+u_2}{2}.
\]

We define trilinear forms $a_h : W_h \times U_h \times U_h \rightarrow \mathbb{R}$ and $b_h : U_h \times F_h \times F_h \rightarrow \mathbb{R}$ by
\begin{align*}
a_h(w,u,v) &= \sum_{K \in \mathcal{T}_h} \int_K w \cdot (v \cdot \nabla u - u \cdot \nabla v) \, dx + \sum_{e \in \mathcal{E}_h} \int_e (n \times \{w\}) \cdot \llbracket u \times v \rrbracket \, ds, \\
b_h(u,f,g) &= \sum_{K \in \mathcal{T}_h} \int_K (u \cdot \nabla f) g \, dx - \sum_{e \in \mathcal{E}_h} \int_e u \cdot \llbracket f \rrbracket \{g\} \, ds,
\end{align*}
The trilinear forms $a_h$ and $b_h$ are well-known: $a_h$ has been used in discretizations the incompressible Euler equations with constant density~\cite{guzman2016h,natale2017variational}, and $b_h$ is a standard discontinous Galerkin discretization of the scalar advection operator~\cite{brezzi2004discontinuous}.  

These trilinear forms possess several properties that play an essential role in the conservative nature of our numerical method.  First, $a_h$ is alternating in its last two arguments:
\[
a_h(w,u,v) = -a_h(w,v,u), \quad \forall (w,u,v) \in W_h \times U_h \times U_h.
\]
Second, using integration by parts, one checks that $b_h$ is also alternating in its last two arguments if its first argument is divergence-free:
\begin{equation} \label{bhplusminus}
b_h(u,f,g) = -b_h(u,g,f), \quad \forall (u,f,g) \in (U_h \cap \mathring{H}(\dv,\Omega)) \times F_h \times F_h.
\end{equation}
As a consequence, we have
\begin{align}
a_h(w,u,u) &= 0, &\forall (w,u) \in W_h \times U_h, \label{auu} \\
b_h(u,f,f) &= 0,  \quad &\forall (u,f) \in (U_h \cap \mathring{H}(\dv,\Omega)) \times F_h. \label{bff} 
\end{align}

We are now ready to define our semidiscrete numerical method.  We first describe the simplest version of the method, which includes no upwinding and exactly preserves total mass, total squared density, total energy, and pointwise incompressibility.  It seeks $u_h(t) \in U_h$, $\rho_h(t) \in F_h$, and $p_h(t) \in Q_h$ such that
\begin{align}
\langle \partial_t (\rho_h u_h), v_h \rangle + a_h(\rho_h u_h, u_h, v_h) - \frac{1}{2} b_h(v_h, \overline{u_h \cdot u_h}, \rho_h) &= \langle p_h, \dv v_h \rangle, & \forall v_h \in U_h, \label{velocityh_cons} \\
\langle \partial_t \rho_h, \sigma_h \rangle - b_h(u_h, \sigma_h, \rho_h) &= 0, & \forall \sigma_h \in F_h, \label{densityh_cons} \\
\langle \dv u_h, q_h \rangle &= 0, &\forall q_h \in Q_h. \label{incompressibleh_cons}
\end{align}

Here, $\overline{f} \in F_h$ denotes the $L^2$-orthogonal projection of $f \in L^2(\Omega)$ onto $F_h$:
\[
\langle \overline{f}, g_h \rangle = \langle f, g_h \rangle, \quad \forall g_h \in F_h.
\]
\begin{remark} \label{remark:uu}
If either
\begin{equation} \label{ordercondition}
( m \ge 2s \text{ and } U_h = RT_s(\mathcal{T}_h) ), \quad  \text{ or } \quad (m \ge 2s+2 \text{ and } U_h = BDM_{s+1}(\mathcal{T}_h) ),
\end{equation}
then $\overline{u_h \cdot u_h} = u_h \cdot u_h$.   Indeed, $u_h$ belongs to the divergence-free subspace of $U_h \in \{RT_s(\mathcal{T}_h),BDM_{s+1}(\mathcal{T}_h)\}$, which consists of piecewise polynomial vector fields of degree at most $s$ (if $U_h = RT_s(\mathcal{T}_h)$) or $s+1$ (if $U_h = BDM_{s+1}(\mathcal{T}_h)$)~\cite[p. 116]{brezzi1991mixed}.  Thus, $\left. u_h \cdot u_h \right|_K \in P_m(K)$ for every $K \in \mathcal{T}_h$ if~(\ref{ordercondition}) holds.  In particular, $\overline{u_h \cdot u_h} = u_h \cdot u_h$ in the lowest-order version of this finite element method ($s=m=0$ and $U_h=RT_0(\mathcal{T}_h)$).  If the condition~(\ref{ordercondition}) is violated, the computation of $\overline{u_h \cdot u_h}$ is not a significant expense since it can be computed element by element.  Note that an analogous projection of the squared fluid velocity appears in other conservative numerical methods for fluids; see, for instance,~\cite[p. 14]{lee2018discrete}.
\end{remark}

As with most finite element methods for advection-dominated problems, solutions of~(\ref{velocityh_cons}-\ref{incompressibleh_cons}) can often exhibit unphysical oscillations if upwinding is not incorporated into the discretization~\cite{ern2004theory}.  If upwinding is desired, we propose the following generalization of~(\ref{velocityh_cons}-\ref{incompressibleh_cons}):   Seek $u_h(t) \in U_h$, $\rho_h(t) \in F_h$, and $p_h(t) \in Q_h$ such that
\begin{align}
\sum_{e \in \mathcal{E}_h} \int_e \Big( \alpha_e(u_h) (n \times \llbracket \rho_h u_h \rrbracket ) \cdot \llbracket u_h \times v_h \rrbracket + \frac{\beta_e(u_h)}{2|u_h \cdot n|^2} (u_h \cdot n)(v_h \cdot n) \llbracket \overline{u_h \cdot u_h}  \rrbracket \cdot & \llbracket \rho_h\rrbracket  \Big) \, ds   \nonumber \\ + \langle \partial_t (\rho_h u_h), v_h \rangle + a_h(\rho_h u_h, u_h, v_h) - \frac{1}{2} b_h(v_h, \overline{u_h \cdot u_h}, \rho_h) - \langle p_h, \dv v_h \rangle  &= 0, \quad\forall v_h \in U_h, \label{velocityh} \\
\langle \partial_t \rho_h, \sigma_h \rangle - b_h(u_h, \sigma_h, \rho_h) + \sum_{e \in \mathcal{E}_h} \int_e \beta_e(u_h) \llbracket \sigma_h \rrbracket \cdot \llbracket \rho_h \rrbracket \, ds
 &= 0, \quad\forall \sigma_h \in F_h, \label{densityh} \\
\langle \dv u_h, q_h \rangle &= 0, \quad\forall q_h \in Q_h. \label{incompressibleh}
\end{align}
Here, $\{\alpha_e\}_{e \in \mathcal{E}_h}$ and $\{\beta_e\}_{e \in \mathcal{E}_h}$ are nonnegative parameters which may depend on $u_h$.  Standard choices for $\alpha_e$ and $\beta_e$ are~\cite{brezzi2004discontinuous,natale2017variational}
\[
\alpha_e(u_h) = c_1 \frac{u_h \cdot n}{|u_h \cdot n|}, \quad \beta_e(u_h) = c_2 |u_h \cdot n|, \quad c_1,c_2 \in \left[0,\tfrac{1}{2}\right],
\]
where the choice $c_1=c_2=\frac{1}{2}$ corresponds to full upwinding.
To understand why, observe that if $\beta_e(u_h) = \frac{1}{2} |u_h \cdot n|$, then 
\[
\left( u_h \{ \rho_h \}  + \beta_e(u_h) \llbracket \rho_h \rrbracket \right) \cdot n_1
= \begin{cases}
\rho_{h1} u_h \cdot n_1 &\mbox{ if } u_h \cdot n_1 > 0, \\
\rho_{h2} u_h \cdot n_1 &\mbox{ if } u_h \cdot n_1 < 0,
\end{cases}
\]
so the upwinding in~(\ref{densityh}) has the effect of replacing $\{\rho_h\}$ with its upwind value (either $\rho_{h1}$ or $\rho_{h2}$) in the expression~(\ref{bhplusminus}) for $b_h(u_h,\sigma_h,\rho_h)$.  The effect of choosing $\alpha_e(u_h) = \frac{1}{2} \frac{u_h \cdot n}{|u_h \cdot n|}$ is similar.  Note that an additional term involving $\beta_e$ has been added to the momentum equation~(\ref{velocityh}) to counteract any energy imbalance introduced by the terms involving $\beta_e$ in the density equation~(\ref{densityh}).  When $\rho_h$ is constant, the terms involving $\beta_e$ vanish, and we recover the momentum upwinding strategy proposed in~\cite{natale2017variational}.

\subsection{Properties of the Spatial Discretization} \label{sec:properties_spatial}

Although~(\ref{incompressibleh}) imposes incompressibility weakly, our choice of finite element spaces~(\ref{FEvelocity}-\ref{FEpressure}) ensures that the velocity field determined by~(\ref{velocityh}-\ref{incompressibleh}) is divergence-free pointwise.
\begin{proposition}
For every $t$, $\dv u_h \equiv 0$.
\end{proposition}
\begin{proof}
Since $U_h \in \{RT_s(\mathcal{T}_h), BDM_{s+1}(\mathcal{T}_h)\}$, we have $\dv u_h \in DG_s(\mathcal{T}_h) = Q_h$, so we may take $q_h = \dv u_h$ in~(\ref{incompressibleh}).
\end{proof}

The following proposition shows that the numerical method~(\ref{velocityh}-\ref{incompressibleh}) exactly preserves total mass, total energy, and, if $\beta_e=0$, total squared density.
\begin{proposition} \label{prop:conservation}
For every $t$, we have
\begin{align}
\frac{d}{dt} \int_\Omega \rho_h \, dx &= 0, \label{rhoint} \\
\frac{d}{dt} \int_\Omega \rho_h^2 \, dx &\le 0, \text{ with equality if } \beta_e = 0, \, \forall e \in \mathcal{E}_h, \label{rho2int} \\
\frac{d}{dt} \int_\Omega \rho_h u_h \cdot u_h \, dx &= 0. \label{energy}
\end{align}
\end{proposition}
\begin{proof}
Taking $\sigma_h \equiv 1$ in~(\ref{densityh}) gives
\[
\frac{d}{dt} \int_\Omega \rho_h \, dx = \langle \partial_t \rho_h, 1 \rangle = b_h(u_h, 1, \rho_h) - \sum_{e \in \mathcal{E}_h} \int_e \beta_e(u_h) \llbracket 1  \rrbracket \cdot \llbracket \rho_h \rrbracket \, ds = 0.
\]
Taking $\sigma_h = \rho_h$ in~(\ref{densityh}) and invoking~(\ref{bff}) gives
\[
\frac{d}{dt} \frac{1}{2} \int_\Omega \rho_h^2 \, dx = \langle \partial_t \rho_h, \rho_h \rangle = b_h(u_h, \rho_h, \rho_h)  -\sum_{e \in \mathcal{E}_h} \int_e \beta_e(u_h) \llbracket \rho_h \rrbracket \cdot \llbracket \rho_h \rrbracket \, ds \le 0.
\]
To prove~(\ref{energy}), we take $v_h=u_h$ in~(\ref{velocityh}) and invoke the definition of the $L^2$-projection to write
\begin{align*}
\frac{d}{dt}  &\frac{1}{2} \int_\Omega \rho_h u_h \cdot u_h \, dx \\&
= \langle \partial_t(\rho_h u_h), u_h \rangle - \frac{1}{2} \langle \partial_t \rho_h, u_h \cdot u_h \rangle \\
&= \langle p_h, \dv u_h \rangle - a_h(\rho_h u_h, u_h, u_h) + \frac{1}{2} b_h(u_h, \overline{u_h \cdot u_h}, \rho_h) - \sum_{e \in \mathcal{E}_h} \int_e \frac{1}{2}\beta_e(u_h) \llbracket \overline{u_h \cdot u_h}  \rrbracket \cdot \llbracket \rho_h\rrbracket \, ds \\&\quad - \frac{1}{2} \langle \partial_t \rho_h, \overline{u_h \cdot u_h} \rangle.
\end{align*}
The first two terms above vanish by~(\ref{incompressibleh}) and~(\ref{auu}).  The last three terms vanish according to the density evolution equation~(\ref{densityh}).  It follows that~(\ref{energy}) holds.
\end{proof}

\section{Temporal Discretization}

We now propose a temporal discretization of~(\ref{velocityh}-\ref{incompressibleh}).  To reduce notational clutter, we supress the subscript $h$ when referring to functions that belong to finite element spaces in this section.  We also suppress the subscript $h$ on $a_h$ and $b_h$.

Our temporal discretization seeks $u_1,u_2,\ldots \in U_h$, $\rho_1,\rho_2,\ldots \in F_h$, and $p_1,p_2,\ldots \in Q_h$ such that for every $k$ and every $(v,\sigma,q) \in U_h \times F_h \times Q_h$,
\begin{align}
&\sum_{e \in \mathcal{E}_h} \int_e \alpha_e(u_{k+1/2}) (n \times \llbracket (\rho u)_{k+1/2} \rrbracket ) \cdot \llbracket u_{k+1/2} \times v \rrbracket \, ds &\nonumber\\
&+ \sum_{e \in \mathcal{E}_h} \int_e \frac{\beta_e(u_{k+1/2})}{2|u_{k+1/2} \cdot n|^2} (u_{k+1/2} \cdot n)(v \cdot n) \llbracket \overline{u_k \cdot u_{k+1}}  \rrbracket \cdot \llbracket \rho_{k+1/2} \rrbracket \, ds &
 \label{velocitydtstab}\\ 
&+ \left\langle \frac{ \rho_{k+1} u_{k+1}-\rho_k u_k }{ \Delta t }, v \right\rangle + a( (\rho u)_{k+1/2}, u_{k+1/2}, v ) - \frac{1}{2}b\left( v, \overline{ u_k \cdot u_{k+1} }, \rho_{k+1/2} \right) - \langle p_{k+1}, \dv v \rangle = 0, \nonumber \\
&\left\langle \frac{\rho_{k+1}-\rho_k}{\Delta t}, \sigma \right\rangle - b( u_{k+1/2}, \sigma, \rho_{k+1/2} ) + \sum_{e \in \mathcal{E}_h} \int_e \beta_e(u_{k+1/2}) \llbracket \sigma \rrbracket \cdot \llbracket \rho_{k+1/2} \rrbracket \, ds = 0, \label{densitydtstab} \\
&\langle \dv u_{k+1}, q \rangle = 0, \label{incompressibledtstab}
\end{align}
where $\Delta t > 0$ is a time step, $u_{k+1/2} = \frac{u_k+u_{k+1}}{2}$, $\rho_{k+1/2} = \frac{\rho_k+\rho_{k+1}}{2}$, and 
\begin{equation} \label{rhouhalf}
(\rho u)_{k+1/2} = \frac{\rho_k u_k + \rho_{k+1} u_{k+1}}{2}.
\end{equation}
In the case where $\alpha_e=\beta_e=0$ for every $e \in \mathcal{E}_h$, the scheme reduces to
\begin{align}
\left\langle \frac{ \rho_{k+1} u_{k+1}-\rho_k u_k }{ \Delta t }, v \right\rangle + a\left( (\rho u)_{k+1/2}, \frac{u_k+u_{k+1}}{2}, v \right) &&\nonumber\\ - \frac{1}{2}b\left( v, \overline{ u_k \cdot u_{k+1} }, \frac{\rho_k+\rho_{k+1}}{2} \right) - \langle p_{k+1}, \dv v \rangle &= 0, & \forall v \in U_h \label{velocitydt} \\
\left\langle \frac{\rho_{k+1}-\rho_k}{\Delta t}, \sigma \right\rangle - b\left( \frac{u_k+u_{k+1}}{2}, \sigma, \frac{\rho_k+\rho_{k+1}}{2} \right) &= 0, & \forall \sigma \in F_h \label{densitydt} \\
\langle \dv u_{k+1}, q \rangle &= 0, & \forall q \in Q_h. \label{incompressibledt}
\end{align}
\begin{remark} \label{remark:uudt}
Just as in Remark~\ref{remark:uu}, we have $\overline{u_k \cdot u_{k+1}} = u_k \cdot u_{k+1}$ if~(\ref{ordercondition}) holds, since then $\left. u_k \cdot u_{k+1} \right|_K \in P_m(K)$ for every $K \in \mathcal{T}_h$.
\end{remark}

\subsection{Properties of the Temporal Discretization}

By our choice of finite element spaces, we again have exact incompressibility of the discrete solution.
\begin{proposition}
For every $k$, $\dv u_k \equiv 0$.
\end{proposition}
\begin{proof}
Take $q = \dv u_{k+1}$ in~(\ref{incompressibledtstab}).
\end{proof}

The next proposition shows that the fully discrete scheme~(\ref{velocitydtstab}-\ref{incompressibledtstab}) enjoys the same conservative properties as the semidiscrete scheme~(\ref{velocityh}-\ref{incompressibleh}).  In particular, it is unconditionally stable for any $\alpha_e,\beta_e \ge 0$, and it exactly preserves total mass, total energy, and, if $\beta_e = 0$ for every $e \in \mathcal{E}_h$, total squared density.

\begin{proposition} \label{prop:conservationdt}
For every $k$, we have
\begin{align}
\int_\Omega \rho_{k+1} \, dx &= \int_\Omega \rho_k \, dx, \label{rhointdt} \\
\int_\Omega \rho_{k+1}^2 \, dx &\le \int_\Omega \rho_k^2 \, dx, \text{ with equality if } \beta_e = 0, \, \forall e \in \mathcal{E}_h, \label{rho2intdt} \\
\int_\Omega \rho_{k+1} u_{k+1} \cdot u_{k+1} \, dx &= \int_\Omega \rho_k u_k \cdot u_k \, dx. \label{energydt}
\end{align}
\end{proposition}
\begin{proof}
Taking $\sigma \equiv 1$ in~(\ref{densitydtstab}) gives
\[
\int_\Omega \frac{\rho_{k+1}-\rho_k}{\Delta t} \, dx = \left\langle \frac{\rho_{k+1}-\rho_k}{\Delta t}, 1 \right\rangle = b\left( \frac{u_k+u_{k+1}}{2}, 1, \frac{\rho_k+\rho_{k+1}}{2} \right) = 0.
\]
Taking $\sigma = \frac{\rho_k+\rho_{k+1}}{2} = \rho_{k+1/2}$ in~(\ref{densitydtstab}) and invoking~(\ref{bff}) gives
\begin{align*}
\frac{1}{2} \int_\Omega \frac{\rho_{k+1}^2-\rho_k^2}{\Delta t} \, dx
&= \left\langle \frac{\rho_{k+1}-\rho_k}{\Delta t}, \frac{\rho_k+\rho_{k+1}}{2} \right\rangle \\
&= b\left( \frac{u_k+u_{k+1}}{2}, \frac{\rho_k+\rho_{k+1}}{2}, \frac{\rho_k+\rho_{k+1}}{2} \right) -\sum_{e \in \mathcal{E}_h} \int_e \beta_e(u_{k+1/2}) \llbracket \rho_{k+1/2} \rrbracket \cdot \llbracket \rho_{k+1/2} \rrbracket \, ds \\
&= -\sum_{e \in \mathcal{E}_h} \int_e \beta_e(u_{k+1/2}) \llbracket \rho_{k+1/2} \rrbracket \cdot \llbracket \rho_{k+1/2} \rrbracket \, ds \\
&\le 0. 
\end{align*}
To prove~(\ref{energydt}), we first observe the identity
\[
\frac{1}{2} \int_\Omega \frac{\rho_{k+1} u_{k+1} \cdot u_{k+1} - \rho_k u_k \cdot u_k}{\Delta t} \, dx
= \left\langle \frac{\rho_{k+1} u_{k+1} - \rho_k u_k}{ \Delta t }, \frac{u_k+u_{k+1}}{2} \right\rangle - \frac{1}{2} \left\langle \frac{\rho_{k+1}-\rho_k}{\Delta t}, u_k \cdot u_{k+1} \right\rangle.
\]
Next, we invoke~(\ref{velocitydtstab}) and the definition of the $L^2$-projection to write
\begin{align*}
&\left\langle \frac{\rho_{k+1} u_{k+1} - \rho_k u_k}{ \Delta t }, \frac{u_k+u_{k+1}}{2} \right\rangle - \frac{1}{2} \left\langle \frac{\rho_{k+1}-\rho_k}{\Delta t}, u_k \cdot u_{k+1} \right\rangle \\
&= \left\langle p_{k+1}, \dv\left( \frac{u_k+u_{k+1}}{2} \right) \right\rangle - a\left( (\rho u)_{k+1/2}, \frac{u_k+u_{k+1}}{2}, \frac{u_k+u_{k+1}}{2} \right) \\ &\quad + \frac{1}{2} b\left( \frac{u_k+u_{k+1}}{2}, \overline{u_k \cdot u_{k+1}}, \frac{\rho_k+\rho_{k+1}}{2}\right) - \frac{1}{2} \left\langle \frac{\rho_{k+1}-\rho_k}{\Delta t}, \overline{u_k \cdot u_{k+1}} \right\rangle \\
&\quad- \frac{1}{2}\sum_{e \in \mathcal{E}_h} \beta_e(u_{k+1/2}) \llbracket \overline{u_k \cdot u_{k+1}} \rrbracket \cdot \llbracket \rho_{k+1/2} \rrbracket \, ds.
\end{align*}
The first two terms above vanish by~(\ref{incompressibledtstab}) and~(\ref{auu}).  The last three terms vanish according to the discrete density evolution equation~(\ref{densitydtstab}).  It follows that~(\ref{energydt}) holds.
\end{proof}

\begin{remark}
The definition~(\ref{rhouhalf}) plays no role in the proof of Proposition~\ref{prop:conservationdt}.  Thus, the conclusions~(\ref{rhointdt}-\ref{energydt}) remain valid if~(\ref{rhouhalf}) is replaced, for instance, by $(\rho u)_{k+1/2} = \left( \frac{\rho_k+\rho_{k+1}}{2}\right) \left(\frac{u_k+u_{k+1}}{2}\right)$, $(\rho u)_{k+1/2} = \rho_k u_k$ or $(\rho u)_{k+1/2} = \rho_{k+1} u_{k+1}$. 
\end{remark}

\section{Numerical Examples}

\begin{table}[t]
\centering
\maketable{variabledensity_a0b0.dat}
\caption{$L^2$-errors in the velocity, density, and pressure at time $T=0.5$ without upwinding.}
\label{tab:a0b0}
\end{table}

\begin{table}[t]
\centering
\maketable{variabledensity_a1b1.dat}
\caption{$L^2$-errors in the velocity, density, and pressure at time $T=0.5$ with upwinding.}
\label{tab:a1b1}
\end{table}

\begin{table}[t]
\centering
\maketabledt{variabledensity_a0b0dt.dat}
\caption{Convergence with respect to $\Delta t$ of the $L^2$-errors in the velocity, density, and pressure at time $T=0.5$ without upwinding.}
\label{tab:a0b0_dt}
\end{table}

\begin{table}[t]
\centering
\maketabledt{variabledensity_a1b1dt.dat}
\caption{Convergence with respect to $\Delta t$ of the $L^2$-errors in the velocity, density, and pressure at time $T=0.5$ with upwinding.}
\label{tab:a1b1_dt}
\end{table}

\subsection{Convergence Tests}

To test the performance of the numerical method~(\ref{velocitydtstab}-\ref{incompressibledtstab}), we used it to approximate the solution of~(\ref{velocity}-\ref{incompressible}) on the square $\Omega = (-1,1) \times (-1,1)$ with initial conditions 
\begin{align*}
u(x,y,0) &= (-\cos(\pi x/2) \sin(\pi y/2), \, \sin(\pi x/2) \cos(\pi y/2)), \\
\rho(x,y,0) &= 2+\sin(xy).
\end{align*}
We used the finite element spaces $U_h = RT_s(\mathcal{T}_h)$ and $F_h = Q_h = DG_s(\mathcal{T}_h)$ with $s \in \{0,1,2\}$ on a uniform triangulation $\mathcal{T}_h$ with maximum element diameter $h = 2^{-j}$, $j=0,1,2,3$.  We used a small time step $\Delta t = 0.00625$ to ensure that temporal discretization errors were negligible, and we measured the $L^2$-error in the computed solution $(u_h,\rho_h,p_h)$ at time $T=0.5$.  The ``exact'' solution was obtained with $s=2$, $h=2^{-5}$, and $\Delta t = 0.00625$.  Tables~\ref{tab:a0b0}-\ref{tab:a1b1} show the results for two choices of the parameters $\alpha_e$, $\beta_e$: $\alpha_e=\beta_e=0$ (no upwinding) and $\alpha_e(u) = \frac{u \cdot n}{2|u \cdot n|}$, $\beta_e(u) = \frac{1}{2}|u \cdot n|$ (upwinding).   When using upwinding, the errors converged at the optimal rates 1, 2, and 3, respectively, for polynomial degrees $s=0,1,2$.  In the absence of upwinding, the convergence rates for $s=1$ were lower (closer to first order than to second order).  These observations are consistent with those of~\cite{guzman2016h,natale2017variational}, where the case of constant density is considered.

Figure~\ref{fig:rhosquared} plots the squared density errors $|1-F(t)/F(0)|$, $F(t) = \int_\Omega \rho_h(t)^2 \, dx$, for the simulations in Tables~\ref{tab:a0b0}-\ref{tab:a1b1} that used $h=\frac{1}{4}$.  As expected, squared density decayed monotonically with upwinding and remained constant (up to roundoff errors) without upwinding.  Total mass and energy errors (not plotted) remained below $10^{-13}$ in these experiments, both with and without upwinding.

To test the convergence of the method with respect to $\Delta t$, we repeated the above experiment with $s=2$, $h=2^{-4}$, and $\Delta t = 2^{-j}$, $j=1,2,3,4$.  The results in Tables~\ref{tab:a0b0_dt}-\ref{tab:a1b1_dt} suggest that the method converges at a second-order rate with respect to $\Delta t$, both with and without upwinding.

\begin{figure}
\centering
\begin{tikzpicture}
\begin{axis}[ymode=log,xlabel=$t$,ylabel=Squared density error,legend style={at={(1.4,-0.25)},
anchor=north east},
  legend columns=3,
  legend style={
    column sep=1ex,
  }] ]
\addplot[blue,mark=o] table [x expr=\coordindex*0.00625, y expr=-\thisrowno{0}]{rhosquaredminus1_variabledensity_a1b1o0d0r2.dat};
\addplot[red,mark=o] table [x expr=\coordindex*0.00625, y expr=-\thisrowno{0}]{rhosquaredminus1_variabledensity_a1b1o1d1r2.dat};
\addplot[black,mark=o] table [x expr=\coordindex*0.00625, y expr=-\thisrowno{0}]{rhosquaredminus1_variabledensity_a1b1o2d2r2.dat};
\addplot[blue] table [x expr=\coordindex*0.00625, y expr=abs(\thisrowno{0})]{rhosquaredminus1_variabledensity_a0b0o0d0r2.dat};
\addplot[red] table [x expr=\coordindex*0.00625, y expr=abs(\thisrowno{0})]{rhosquaredminus1_variabledensity_a0b0o1d1r2.dat};
\addplot[black] table [x expr=\coordindex*0.00625, y expr=abs(\thisrowno{0})]{rhosquaredminus1_variabledensity_a0b0o2d2r2.dat};
\legend{$s=0$ (upwind),$s=1$ (upwind),$s=2$ (upwind),$s=0$ (no upwind),$s=1$ (no upwind),$s=2$ (no upwind)}
\end{axis}
\end{tikzpicture}
\caption{Squared density errors $|1-F(t)/F(0)|$, $F(t) = \int_\Omega \rho_h(t)^2 \, dx$, for the simulations in Tables~\ref{tab:a0b0}-\ref{tab:a1b1} that used $h=\frac{1}{4}$.  Total mass and energy errors (not plotted) remained below $10^{-13}$ in these experiments, both with and without upwinding.}
\label{fig:rhosquared}
\end{figure}

\subsection{Rayleigh-Taylor Instability}

\begin{figure}
\centering
\includegraphics[scale=0.3,trim=9in 0in 9in 0in,clip=true]{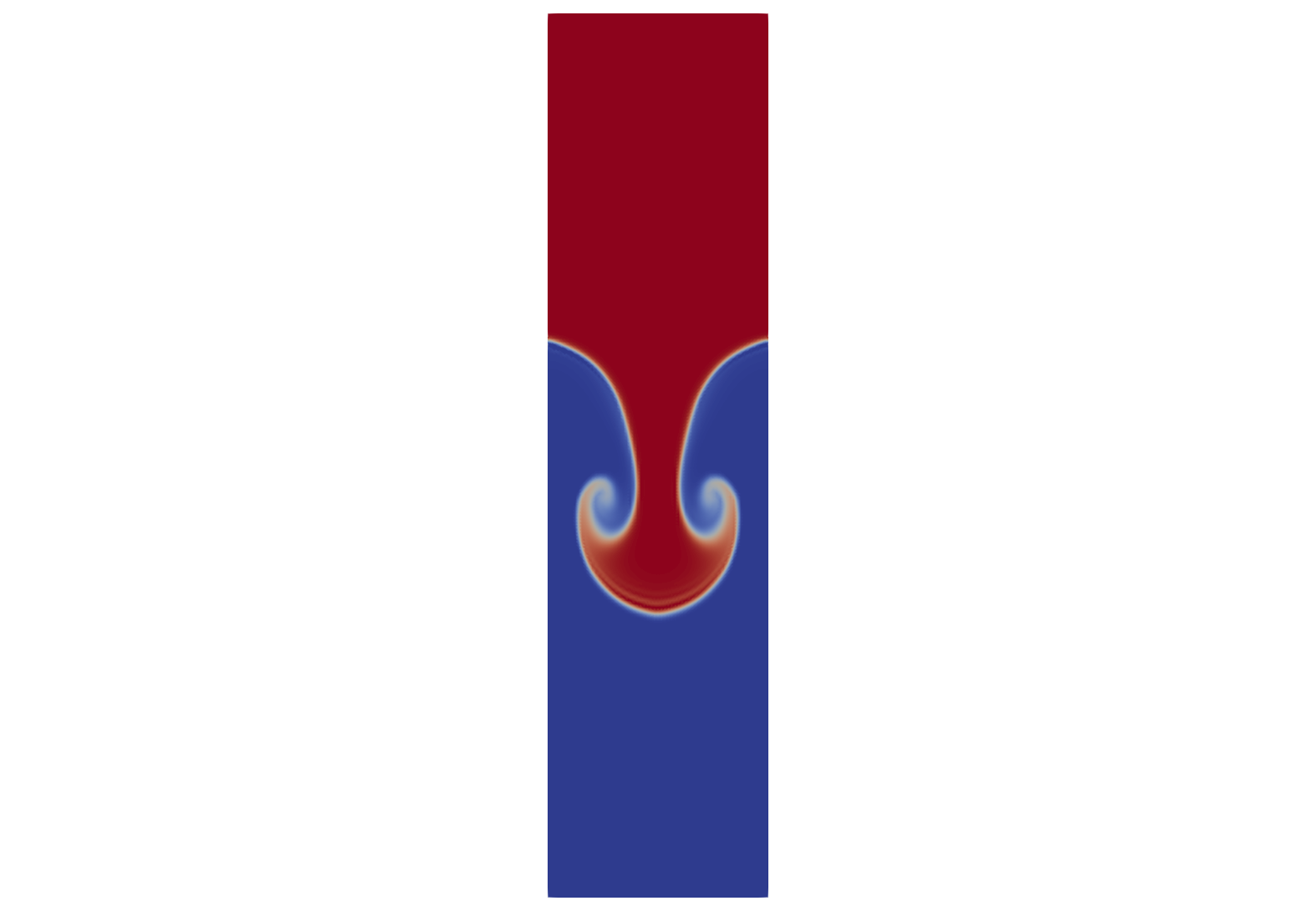}
\includegraphics[scale=0.3,trim=9in 0in 9in 0in,clip=true]{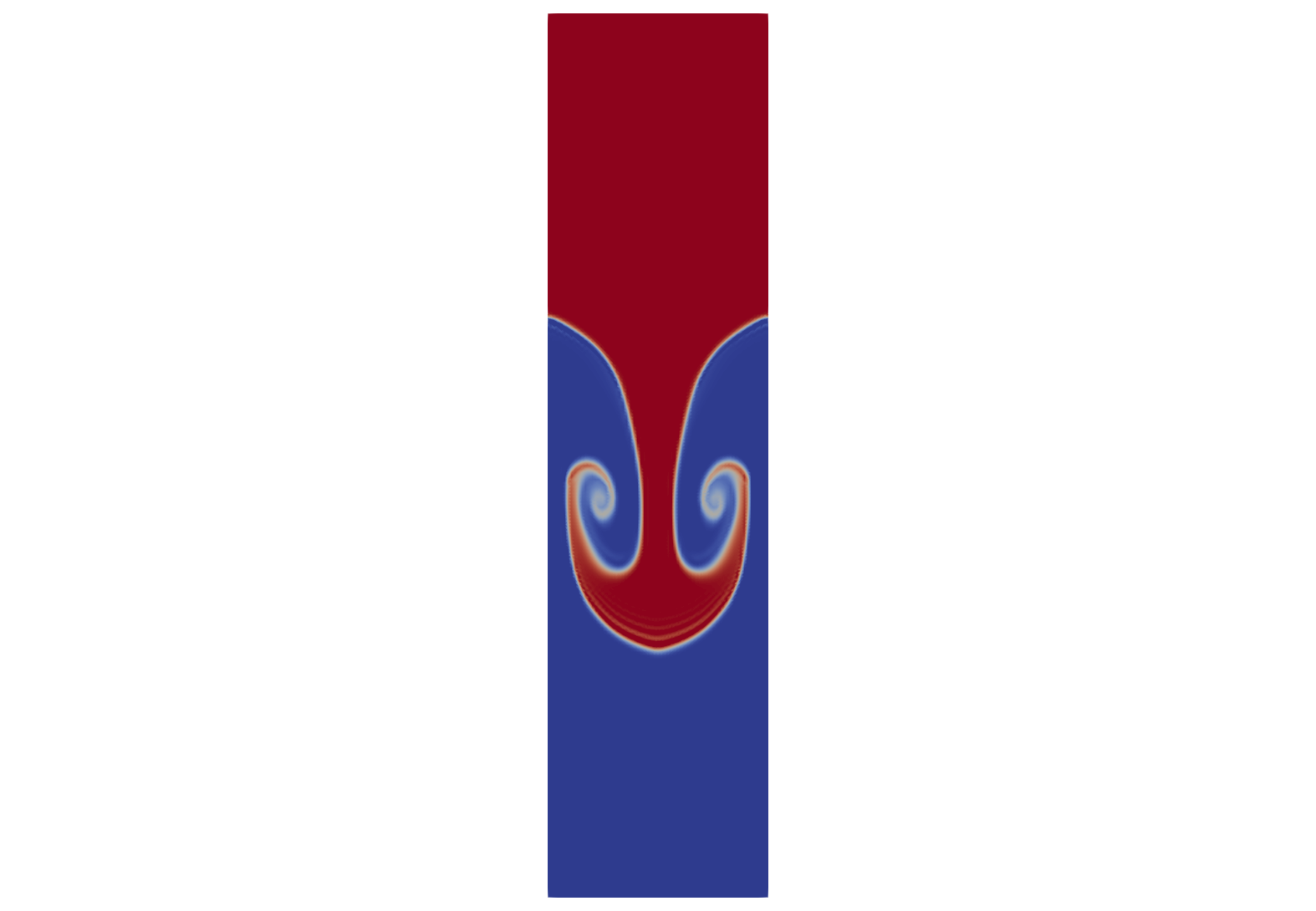}
\includegraphics[scale=0.3,trim=9in 0in 9in 0in,clip=true]{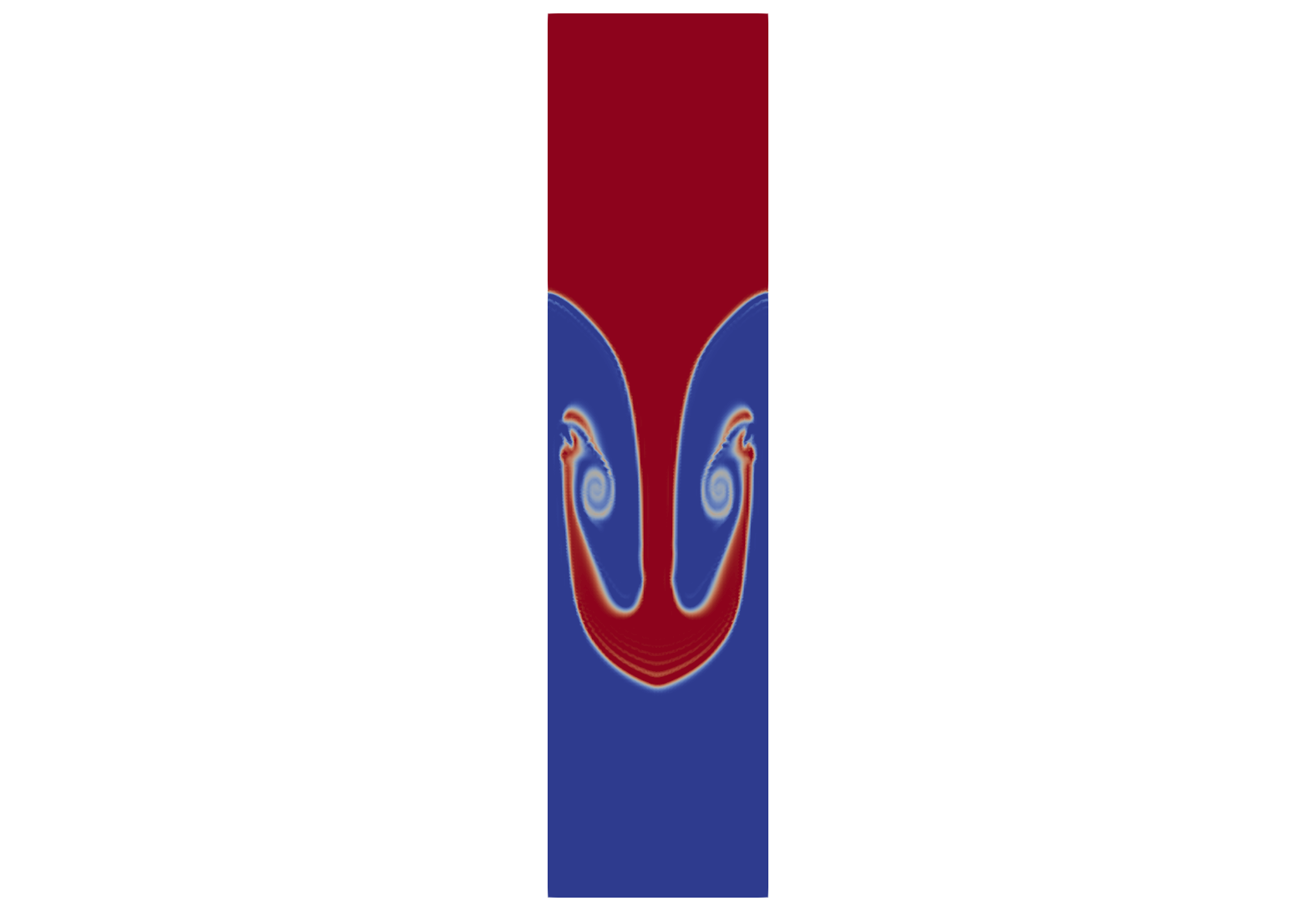}
\includegraphics[scale=0.3,trim=9in 0in 9in 0in,clip=true]{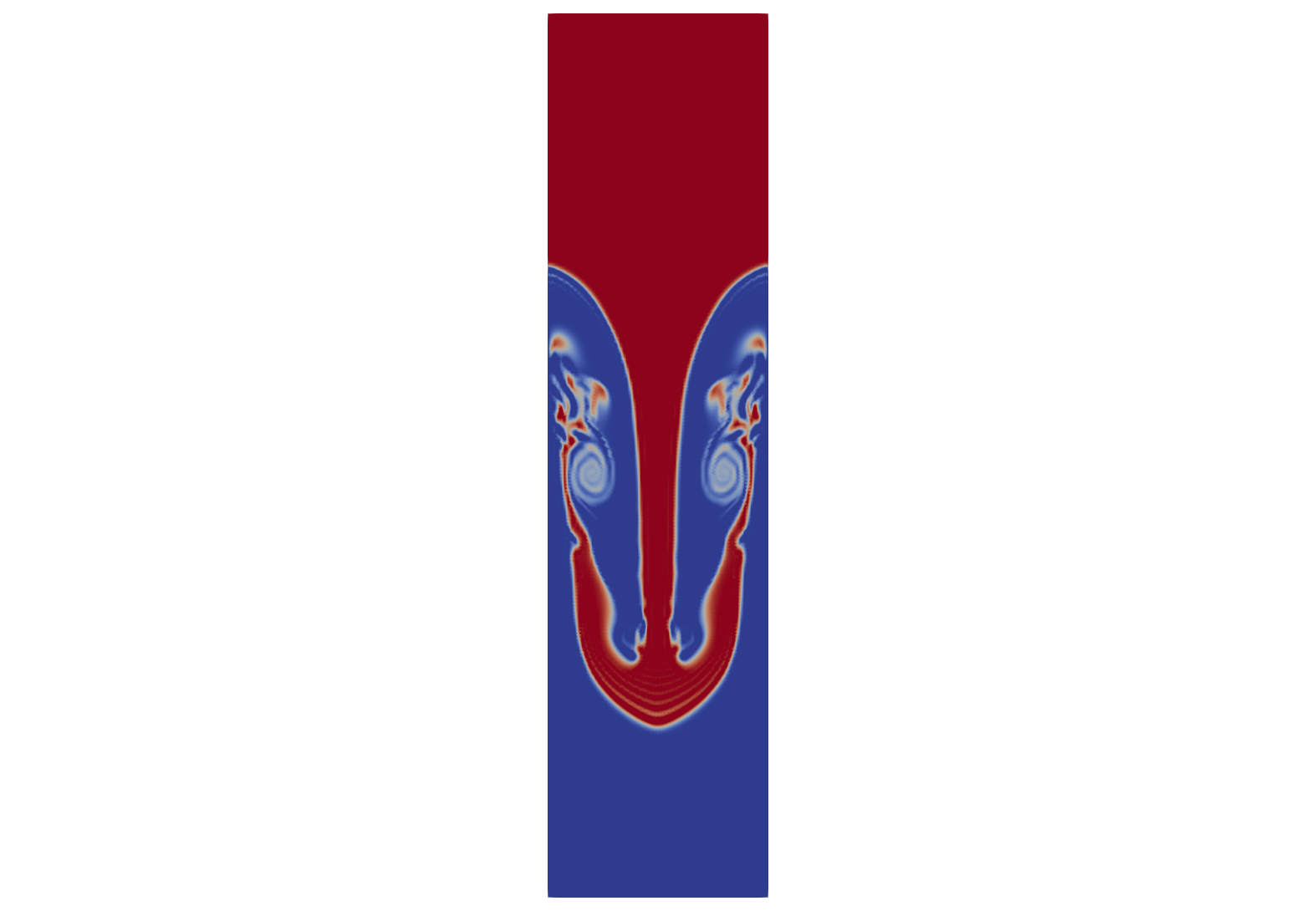}
\caption{Density contours at $t=0.8,0.95,1.1,1.25$ in the Rayleigh-Taylor instability simulation.}
\label{fig:raytay}
\end{figure}

As a second test, we simulated the Rayleigh-Taylor instability on a rectangle $\Omega = (-1/2,1/2) \times (2,2)$ with initial conditions
\begin{align*}
u(x,y,0) &= (0,0), \\
\rho(x,y,0) &= 2+\tanh\left( \frac{y+0.1 \cos(2\pi x)}{0.1} \right).
\end{align*}
For this test, we added a gravitational forcing term $\langle (0,-10)\rho_{k+1/2}, v \rangle$ to right-hand side of the momentum equation~(\ref{velocitydtstab}).  We used the upwind scheme with $\Delta t = 0.01$ and with $U_h = RT_0(\mathcal{T}_h)$, $F_h = DG_1(\mathcal{T}_h)$, and $Q_h = DG_0(\mathcal{T}_h)$ on a uniform triangulation $\mathcal{T}_h$ with maximum element diameter $h = 0.015625$.  Plots of the density at various times $t$ are shown in Figure~\ref{fig:raytay}.

\section{Conclusion}

We have constructed a numerical method for the incompressible Euler equations with variable density that exactly preserves total mass, total squared density, total energy, and pointwise incompressibility at the spatially and temporally discrete levels.  The method achieves second-order accuracy in time, and allows for the use of high-order finite elements to achieve high-order accuracy in space.  An upwinded version of the method was also described and proved to be stable.  Numerical tests illustrated its convergence order and its performance on a simulation of the Rayleigh-Taylor instability.

\section{Acknowledgements}

E.G. was partially supported by NSF grant DMS-1703719.  FGB was partially supported by the ANR project GEOMFLUID, ANR-14-CE23-0002-01.

\bibliographystyle{abbrv}
\bibliography{references}

\begin{thebibliography}{1}

\bibitem{brezzi1991mixed}
F.~Brezzi and M.~Fortin.
\newblock {\em Mixed and Hybrid Finite Element Methods}, volume~15.
\newblock Springer-Verlag, 1991.

\bibitem{brezzi2004discontinuous}
F.~Brezzi, L.~D. Marini, and E.~S{\"u}li.
\newblock Discontinuous {G}alerkin methods for first-order hyperbolic problems.
\newblock {\em Mathematical models and methods in applied sciences},
  14(12):1893--1903, 2004.

\bibitem{ern2004theory}
A.~Ern and J.-L. Guermond.
\newblock {\em Theory and Practice of Finite Elements}, volume 159.
\newblock Springer Science \& Business Media, 2004.

\bibitem{gawlik2019variational}
E.~S. Gawlik and F.~Gay-Balmaz.
\newblock A variational finite element discretization of compressible flow.
\newblock {\em Preprint}, 2019.

\bibitem{guermond2000projection}
J.-L. Guermond and L.~Quartapelle.
\newblock A projection {FEM} for variable density incompressible flows.
\newblock {\em Journal of Computational Physics}, 165(1):167--188, 2000.

\bibitem{guzman2016h}
J.~Guzm{\'a}n, C.-W. Shu, and F.~A. Sequeira.
\newblock H(div) conforming and {DG} methods for incompressible euler's
  equations.
\newblock {\em IMA Journal of Numerical Analysis}, 37(4):1733--1771, 2016.

\bibitem{lee2018discrete}
D.~Lee, A.~Palha, and M.~Gerritsma.
\newblock Discrete conservation properties for shallow water flows using mixed
  mimetic spectral elements.
\newblock {\em Journal of Computational Physics}, 357:282--304, 2018.

\bibitem{natale2017variational}
A.~Natale and C.~J. Cotter.
\newblock A variational finite-element discretization approach for perfect
  incompressible fluids.
\newblock {\em IMA Journal of Numerical Analysis}, 38(3):1388--1419, 2017.

\end{thebibliography}

\end{document}